\documentclass[11pt,a4paper,makeidx, amstex]{amsart}
\usepackage[latin1]{inputenc}        
\usepackage[dvips]{graphics}

\usepackage
{graphicx}
\usepackage{epstopdf}
\usepackage{mathrsfs}
\let\mathcal\mathscr
\usepackage{color}		
\usepackage{epsfig}
\psdraft

\usepackage{amssymb}
\usepackage{amsmath}
\usepackage{amsfonts}
\usepackage{latexsym}
\usepackage[T1]{fontenc}
\usepackage[latin1]{inputenc}
\usepackage{hyperref}
\usepackage[all,ps]{xy}
\RequirePackage{amsthm}
\RequirePackage{amssymb}
\usepackage[all,ps]{xy}
\usepackage{latexsym}
\usepackage{amscd}
\usepackage[latin1]{inputenc}
\usepackage[T1]{fontenc}
\usepackage{color}
\RequirePackage{amsthm}
\RequirePackage{amssymb}
\usepackage[all,ps]{xy}
\usepackage{latexsym}

\def\Z{{\mathbb Z}}

\def\Q{{\mathbb Q}}

\def\Im{\mathop{\rm Im}\nolimits}

\def\Mon{\mathop{\rm Mon}\nolimits}

\def\Pic{\mathop{\rm Pic}\nolimits}
\def\supp{\mathop{\rm Supp}\nolimits}

\def\tilde{\widetilde}

\def\phi{\varphi}

\def\Pic{\mathop{\rm Pic}\nolimits}
\def\dim{\mathop{\rm dim}\nolimits}

\newcommand{\ra}{\rightarrow}

\newtheorem{thm}{Theorem}[section]
\newtheorem{defn}[thm]{Definition}

\newtheorem{cor}[thm]{Corollary}
\newtheorem{rmk}[thm]{Remark}
\newtheorem{prop}[thm]{Proposition}
\newtheorem{lem}[thm]{Lemma}

\title[Polarized parallel transport and uniruled divisors]{Polarized parallel transport and uniruled divisors on deformations of generalized Kummer varieties}

\author{Giovanni Mongardi}
\address{Dipartimento di Matematica, Universit\'a degli studi di Milano, via Cesare Saldini 50 - 20121 Milano, Italy}
\email{giovanni.mongardi@unimi.it}
\author{Gianluca Pacienza}
\address{Institut de Recherche Math\'ematique Avanc\'ee,
Universit\'e de Strasbourg et CNRS,
7, Rue R. Descartes - 67084 Strasbourg CEDEX, France}
\email{pacienza@math.unistra.fr}
\date{\today}

\begin{document}

%
\begin{abstract}
In this note we characterize polarized parallel transport operators on irreducible holomorphic symplectic varieties which are deformations of generalized Kummer varieties. We then apply such characterization to show the existence of ample uniruled divisors on these varieties and derive some interesting consequences on their Chow group of 0-cycles.
\end{abstract}
%
%
%
%
\maketitle
{\let\thefootnote\relax
\footnote{
\noindent
\textbf{A.M.S.~classification :} 14C99, 14J28, 14J35, 14J40. \\
\noindent
G.M. is supported by FIRB 2012 ``Spazi di Moduli e applicazioni''. G.P. was partially supported by the University of Strasbourg Institute for Advanced Study (USIAS) as part of a USIAS Fellowship. The project was partially funded by GDRE ``GRIFGA''. 
 }}
\numberwithin{equation}{section}

%
\section{Introduction}
%
The purpose of this note is to show that the techniques and results from \cite{Markman09} can be extended, using \cite{mon}, to study  polarized parallel transport operators (cf. Definition \ref{def:op} below) on projective irreducible holomorphic symplectic varieties which are deformations of generalized Kummer varieties (we call such varieties ``of $K_n(A)$-type'').   Moreover, such a study can be successfully applied, along the lines of \cite{CP}, to prove the existence of ample uniruled divisors on these varieties and derive interesting consequences on their Chow group of 0-cycles, thus unifying the picture for the two deformation types known to exist in all dimension. The key object is a monodromy invariant function $f_{X}$ that we define (cf. Section \ref{sec:mon}) from divisors on any $X$ of $K_n(A)$-type which satisfies the following. 

\begin{thm}\label{thm:parallel}
Let $(X_1,H_1)$ and $(X_2,H_2)$ be two primitively polarized irreducible holomorphic symplectic variety of $K_n(A)$-type such that the square of the polarizations with respect to the Beauville-Bogomolov forms $q_{X_1}$ and $q_{X_2}$ are the same $q_{X_1}(H_1)=q_{X_2}(H_2)$. Then $f_{X_1}(c_1(H_1))=f_{X_2}(c_1(H_2))$ if and only if there exists a polarized parallel-transport operator from $(X_1,H_1)$ to $(X_2,H_2)$.
\end{thm}

Thanks to the control of the deformation theory of rational curves covering a divisor inside irreducible holomorphic symplectic varieties (cf. \cite[Corollary 3.5]{CP}) and to the existence of many uniruled divisors inside generalized Kummer varieties (cf. Section \ref{sec:ex}), we can apply this result to prove the following:

\begin{thm}\label{thm:existence}
Any ample linear system on any irreducible holomorphic symplectic variety of $K_n(A)$-type
contains a member which is a sum of uniruled divisors. 
\end{thm}

As shown in \cite[Theorems 5.1 and 5.4]{CP}, the existence of uniruled divisors in any ample linear system implies the following. 
\begin{cor}\label{cor:chow} Let $X$ be a projective irreducible holomorphic symplectic variety of $K_n(A)$-type.
\begin{enumerate}
\item[(i)] For any $D_1,D_2\subset X$ irreducible and uniruled divisors we have
$$
\Im \big( CH_0(D_1)_\Q \to CH_0(X)_\Q \big)= \Im \big( CH_0(D_2)_\Q \to CH_0(X)_\Q\big)=: S_1(X)_\Q
$$
\item[(ii)] If $\Pic(X)\otimes CH_1(X)\to   CH_0(X)$ denotes the natural intersection map, we have 
$$
\Im \big(\Pic(X)\otimes CH_1(X)_\Q\to   CH_0(X)_\Q\big)= S_1(X)_\Q.
$$
 \end{enumerate}
\end{cor}
Corollary \ref{cor:chow} represents a generalization to projective irreducible holomorphic  symplectic varieties of $K_n(A)$-type of the Beauville-Voisin results \cite{BV} saying that (i) any point on any rational curve on a projective $K3$ surface $X$ determines the same class $c_X$ inside $CH_0(X)$ and (ii) the intersection product $\Pic(X)\otimes \Pic(X)\to   CH_0(X)$ takes values in the multiples of $c_X$. 

The motivation for our result comes from Beauville's conjectural reinterpretation \cite{B07} of the above results from \cite{BV}, known under the name of  ``weak splitting property conjecture'' and by the more recent \cite{V15}, where Voisin suggests a strategy to study the conjectural splitting of the Bloch-Beilinson filtration based on the (conjectural) existence of codimension $i$ subvarieties covered by $i$-dimensional constant cycle subvarieties (see \cite[Conjecture 0.4]{V15}). Theorem \ref{thm:existence} proves this conjecture for $i=1$ and $X$ of  $K_n(A)$-type. Closely related results are those of Lin \cite{Lin}, who proves Voisin's conjecture when $X=K_n(A)$, and those of Fu \cite{Fu} who shows that the Beauville weak splitting property conjecture holds when $X=K_n(A)$.

%
\section{Preliminaries}
%
Let $A$ be an abelian surface. Let $A^{[n+1]}$ denote its Hilbert scheme of length $n+1$ zero dimensional subschemes of $A$. It has dimension $2n+2$ and its Albanese map surjects onto $A$ itself. Let $K_n(A)$ denote a fibre of the Albanese map. This manifold is called generalized Kummer variety of $A$ and it is an irreducible holomorphic symplectic manifold, i.e. a simply connected compact K\"ahler manifold such that the $H^{2,0}(X)$ is generated by a symplectic form.
Moreover, a  generalized Kummer variety has $b_2=7$. We denote by $\tilde{\Lambda}$ the weight two Hodge structure on $H^{2*}(A,\mathbb{Z})$ with bilinear form given by the intersection pairing and $(2,0)$ part given by $H^{2,0}(A)$. Its elements are usually denoted $(r,D,s)$, where $r\in H^0,\,s\in H^4$, $D\in H^2$ and $(r,D,s)^2=D^2-2rs$.  As a lattice, it is isometric to $U^4$, where as usual 
$U$ is the hyperbolic lattice $\begin{pmatrix}
0& 1\\
1& 0
\end{pmatrix}.$
 The second integral cohomology of $K_n(A)$ can be embedded in $\tilde{\Lambda}$ in a canonical way and, under this embedding, we have $H^2(A)\subset H^2(K_n(A))$ and $H^2(K_n(A))^\perp=(2n+2)$, where $(2n+2)$ denotes a one dimensional lattice generated by an element of square $2n+2$. Thus, $H^2(K_n(A))\cong U^3\oplus (-2n-2)$. 
The restriction $\Delta_n$ to $K_n(A)$ of the exceptional divisor of the Hilbert-Chow morphism $A^{[n+1]}\to A^{(n+1)}$ has class divisible by two and if we write $\Delta_n=2e_n$ we have the following decomposition
$$
H^2(K_n(A),\mathbb Z)=H^2(A,\mathbb Z) \oplus \mathbb Z e_n
$$
 which is orthogonal with respect to the Beauville-Bogomolov form on $H^2(K_n(A),\mathbb Z)$. Moreover, we have $q(e_n)=-2(n+1)$. Dually, we have
 $$
  N_1(K_n(A))= N_1(A) \oplus \mathbb Z r_n
 $$
where $r_n$ is the class of an exceptional rational curve, i.e. a general fiber of the Hilbert-Chow morphism.

Let $L$ be an even lattice, and let $A_L:=L^\vee/L$ be its discriminant group. It inherits a quadratic form with values in $\mathbb{Q}/2\mathbb{Z}$ from the quadratic form on $L$. 
Recall that the { divisibility} $div(v)$ of an element $v\in L$ is the positive generator 
of the ideal $v\cdot L \subset \mathbb Z$. We will denote by $[v/div(v)]$ the class in the discriminant group $A_L$ of the (primitive) element $v/div(v)\in L^\vee$.

When the lattice $L$ contains two copies of the hyperbolic lattice $U$, the discriminant form governs most of the properties of $L$ and, in particular, we have the following isometry criterion, known as Eichler's criterion.

\begin{lem}[\cite{GriHulSan10}, Lemma 3.5]\label{lem:eichler}
Let $L'$ be an even lattice and let $L=U^2\oplus L'$. Let $v,w\in L$ be two primitive elements such that the following holds:
\begin{itemize}\renewcommand{\labelitemi}{$\bullet$}
\item $v^2=w^2$.
\item $[v/div(v)]=[w/div(w)]$ in $A_L$.
\end{itemize}
Then there exists an isometry $g\in \tilde{O}^+(L)$ with determinant one and such that $g(v)=w$.
\end{lem}
Here, $O^+$ denotes isometries preserving the orientation on $L$ and $\tilde{O}$ denotes isometries whose induced action on $A_L$ is trivial. Moreover, the isometry given by the Lemma is a composition of three basic isometries, called Eichler's transvections, which have determinant one.
The above Lemma is paramount, due to the following characterization of the monodromy group of a manifold $X$ of $K_n(A)$-type, which can be found in \cite{mon}. Recall that the monodromy group $\Mon^2(X)$ is the group of isometries of $H^2(X,\mathbb Z)$ obtained by parallel transport.
Let $\chi$ be the character of $O(H^2(X,\mathbb{Z}))$ on the discriminant group $A_X:=H^2(X,\mathbb{Z})^\vee/H^2(X,\mathbb{Z})$. More explicitely, if $T\in O(H^2(X,\mathbb{Z}))$ and $[\phi] \in A_X$, then $T [\phi]$ is the element of $A_X$ sending $v\in H^2(X,\mathbb{Z})$ to the class of $\phi (T v)$.
 Let $det$ be the determinant character.  
\begin{thm}[\cite{mon}, Theorem 2.3]\label{thm:mon}
Let $X$ be a holomorphic symplectic manifold of $K_n(A)$-type. 
Then $\Mon^2(X)=(det\cdot \chi)^{-1} (1)$.
\end{thm}
The above theorem implies that all monodromy operators act as $\pm 1 $ on $A_X$.
\begin{rmk}\label{rem:discri}
{\em 
The discriminant group $A_X$ is generated by a single element of order $2n+2$ and square $\frac{-1}{2n+2}$. If $X=K_n(A)$, this is rather easy to see by using the previous embedding $H^2(X,\mathbb{Z})\rightarrow \tilde{\Lambda}$, whose orthogonal is generated by the Mukai vector $(1,0,-n-1)$ and the class $e_n$ is sent to $(1,0,n+1)$.
Furthermore, by \cite[Proposition 1.4.1]{Nik}, the discriminant group is isomorphic to the discriminant group of any unimodular complement, such as the rank one lattice generated by $v_n:=(1,0,-n-1)$, which has a single generator.

 Clearly, the element
$$
(\frac{1}{2n+2},0,1)
$$ 
has integral pairing with $H^2(K_n(A),\mathbb{Z})$ and pairing one with $e_n$, so its class is primitive in the discriminant group, hence it is a generator.  
}
\end{rmk}

\begin{rmk}\label{remark:projective-def-classes}
{\em As proven in \cite{marmeh}, holomorphic symplectic manifolds of the form $K_n(A)$ are dense in their own deformation space and the same holds for polarized $K_n(A)$ inside their space of polarized deformations, see \cite{MP2}.}
\end{rmk}

%
\section{Characterization of parallel transport operators}\label{sec:mon}
%
The aim of the present section is to prove Theorem \ref{thm:parallel}. Its proof consists in two steps: we first define a function as in the theorem and we prove that this function is a monodromy invariant, see Lemma \ref{lem:moninv}. Then, we prove that this function is easily computable on moduli spaces of sheaves on abelian surfaces (see Lemma \ref{lem:cambioA}), which will be used in our applications. We start by recalling the following. 

\begin{defn}\label{def:op} Let $(X_1,H_1)$ and $(X_2,H_2)$ be two polarized irreducible holomorphic symplectic varieties. 
\begin{itemize}
\item[(i)] A family $(\mathcal X, \mathcal H) \to B$ of polarized irreducible holomorphic symplectic varieties is a smooth and proper family $\mathcal X \to B$ of irreducible  holomorphic symplectic varieties together with a line bundle $\mathcal H$ on $\mathcal X$ which restricts to an ample line bundle $H_b$ on $X_b$ for every $b\in B$.

\item[(ii)]
The (polarized) irreducible holomorphic symplectic varieties $(X_1,H_1)$ and $(X_2,H_2)$
are (polarized) deformation equivalent if there exists a  family $(\mathcal X, \mathcal H) \to B$ of (polarized) irreducible holomorphic symplectic varieties and two points $b_1,b_2\in B$ such that 
$(X_1,H_1)=(X_{b_1}, H_{b_1})$ and $(X_2,H_2)=(X_{b_2}, H_{b_2})$.

\item[(iii)] An isomorphism $f:H^2(X_1,\mathbb Z)\to H^2(X_2,\mathbb Z)$ between the second cohomology groups of two (polarized) irreducible holomorphic symplectic varieties which are deformation equivalent is a (polarized) parallel transport operator if  there exists a smooth and proper family $\pi: \mathcal X \to B$ of  (polarized) irreducible holomorphic symplectic varieties and a continuous path $\gamma: [0,1]\to B$ satisfying $\gamma(0)=b_1, \gamma(1)=b_2$ such that the isomorphism $f$ is induced by the parallel transport in the local system $R\pi_\star \mathbb Z$ (and $f(H_1)=f(H_2)$).

\end{itemize}
\end{defn}

We recall the definition of faithful monodromy invariant from \cite{Markman09}. Let $I(X)$ be a subset of $H^2(X,\mathbb{Z})$ which is invariant under the action of $\Mon^2(X)$. Let $\Sigma$ be a set and let $f\,:\,I(X)\,\rightarrow\,\Sigma$ be a function. 
\begin{defn}
The function $f$ is a monodromy invariant if it is surjective and $f(h)=f(g(h))$ for all $g\in \Mon^2(X)$. If in addition the induced function $f$ on $I(X)/\Mon^2(X)$ is injective, then we call it a faithful monodromy invariant.
\end{defn}
For our purposes, we can take the above set $I(X)$ to be invariant under the action of the full isometry group $O(H^2(X,\mathbb{Z}))$. Given a family of smooth deformations of $X$, it is possible to extend a monodromy invariant $f$ by parallel transport to the whole family.  We keep denoting such an extension by $f$. We then have the following.
\begin{lem}[\cite{Markman09}, Lemma 5.17]\label{lem:mark}
Let $f\,:\,I(X)\,\rightarrow\,\Sigma$ be a faithful monodromy invariant function and let $(X_i,h_i)$ be two pairs such that $X_i$ is deformation equivalent to $X$ and $h_i\in I(X_i)$. Then the following holds:
\begin{itemize}
\item The pairs $(X_1,h_1)$ and $(X_2,h_2)$ are deformation equivalent (in the sense of 
\cite[Definition 5.6]{Markman09}) if and only if $f(h_1)=f(h_2)$.
\item Suppose $h_1=\langle c_1(L) \rangle$ for some $L\in Pic(X_1)$ and suppose $(L,\omega)>0$ for some K\"ahler class $\omega$. Then $f(h_1)=f(h_2)$ if and only if they are deformation equivalent and the deformation of $L$ is a line bundle with the same properties.
\end{itemize}
\end{lem}
We now proceed to define a faithful monodromy invariant function and we compute it on a dense subset of the moduli space of generalized Kummer varieties, namely for all those manifolds who are moduli spaces of stable objects in the derived category of an abelian surface. See Remark \ref{oss:embedding} for a discussion on what is needed to compute this invariant on an arbitrary deformation.\\
Let $v=(r,D,s)\in H^{2*}(A,\mathbb{Z})$ be a primitive Mukai vector and let $X=M_{v,\sigma}(A)$ be a smooth moduli space of stable objects of Mukai vector $v$ in the derived category of an abelian surface $A$, as defined in \cite{yos12}. Let $I(X)$ be an isometry invariant subset of $H^2(X,\mathbb{Z})\cong v^\perp\subset H^{2*}(A,\mathbb{Z})$. Let $F_{v,X}$ be the function which sends an element  $w\in I(X)$ to the saturation $T$ of the lattice containing $w$ and $v$. Let $\Gamma_v$ be the group of isometries of $H^{2*}(A,\mathbb{Z})$ which fix $v$, preserve the orientation on $v^\perp$ and have determinant $1$. Let $\Sigma_{v,X}$ be the quotient of $F_{v,X}(I(X))$ by $\Gamma_{v}$. Let $f_{v,X}\,:\,I(X)\,\rightarrow\,\Sigma_{v,X}$ be the function induced by $F_{v,X}$.
\begin{lem}\label{lem:moninv}
The function $f_{v,X}$ is a faithful monodromy invariant.
\end{lem} 
\begin{proof}
Let $g\in \Mon^2(X)$. Let $\chi$ be the characther of its action on the discriminant group $A_X$. Recall that by Theorem \ref{thm:mon} we have $\chi=\pm 1$. Therefore, $g$ can be extended to an isometry $\tilde{g}$ of $H^{2*}(A,\mathbb{Z})$ as follows:
\begin{align}\nonumber
\tilde{g}(t) &:= g(t)\,\,\,\text{for } t\in H^2(X,\mathbb{Z})\cong v^\perp \\ \nonumber
\tilde{g}(v) &:= \chi(g) v.
\end{align}
As, again by Theorem \ref{thm:mon}, we have $det(g)\cdot \chi(g)=1$, then the isometry $\tilde{g}$ lies in $\Gamma_{v}$. Thus, $F_{v,X}(g(T))=\tilde{g}(F_{v,X}(T))$, i.e. $f_{v,X}(T)=f_{v,X}(g(T))$.
\end{proof}
\begin{rmk}\label{rmk:1}
{\em Notice that if $X$ and $X'$ are deformation equivalent irreducible holomorphic symplectic varieties of $K_n(A)$-type and $f_X$ is a monodromy invariant function for $X$, then the extension via parallel transport of $f_X$ to $X'$ is a monodromy invariant function for $X'$.} 
\end{rmk}

\begin{lem}\label{lem:cambiov}
Let $M_v(A)\cong X \cong M_w(A')$. Then there is a canonical bijection between $\Sigma_{v,X}$ and $\Sigma_{w,X}$ which identifies $f_{v,X}$ and $f_{w,X}$ 
\end{lem}
\begin{proof}
First of all, the abelian surfaces $A$ and $A'$ are derived equivalent by \cite[Proposition 3.15]{yos12}, see also \cite[Proposition 2.4]{MW}. Moreover, we have a Hodge isometry $i_{v,w}$ between $H^{2*}(A)$ and $H^{2*}(A')$ which sends $v$ in $w$ and thus restricts to an isometry between $v^\perp\cong H^2(M_v(A))$ and $w^\perp\cong H^2(M_w(A'))$.
The map $i_{v,w}$ satisfies the following:
\begin{align}
i_{v,w}\Gamma_v&=\Gamma_w i_{v,w}\\
i_{v,w}(T_v(H))&=T_w(i_{v,w}(H)),\,\,\forall\,H\in I(X) 
\end{align} 
Here, the first equation is just the identification $v^\perp\cong H^2(M_v(A))=H^2(M_w(A'))\cong w^\perp$, and the second equality is given by the Hodge isometry $H^{2*}(A)\cong H^{2*}(A')$.

Therefore, it induces a bijection between $\Sigma_v$ and $\Sigma_w$ which identifies $f_{v,X}$ and $f_{w,X}$.
\end{proof}
\begin{rmk}\label{rmk:2}
{\em The previous lemma shows that the monodromy invariant does not depend on the way we realize $X$ as a moduli space. }
\end{rmk}
\begin{lem}\label{lem:cambioA}
Let $X=M_v(A)$ and $Y=M_w(A')$ be deformation equivalent. Let $f$ be the extension by deformation of $f_{v,X}$. Then $f$ and $f_{w,Y}$ coincide on $I(Y)$.
\begin{proof}
By \cite[Section 5.3]{Markman09}, the function $f_{v,X}$ admits a unique extension which we call $f$.
Let $X'=M_v(B)$ be a deformation of $X$ obtained by deforming the underlying abelian surface $A$ to $B$. By continuity, the monodromy invariant $f$ gives $f_{v,X'}$ on $X'$. Let now $\gamma$ be a path in the deformation space of $X$ which connects $X$ and $Y$ and let us suppose that all points of $\gamma$ correspond to moduli spaces on abelian surfaces. This can be done, as moduli spaces are a connected dense subset of the total deformation space (see \cite[Proposition 2.31]{PR}). Therefore, the deformation $\gamma$ is a composition of deformations of the underlying surfaces with isomorphisms as in Lemma \ref{lem:cambiov}. Both operations identify the global function $f$ with the local function $f_{v',Z}$ for all $Z\in \gamma$, therefore our claim follows.   
\end{proof}
\end{lem}
\begin{rmk}\label{rmk:3}
{\em If $V$ is an irreducible holomorphic symplectic variety deformation of a smooth moduli space $M_v(A)$ (of stable objects on an abelian surface), then 
the previous lemma shows that the monodromy invariant of $V$ defined via parallel transport by the  monodromy invariant of $M_v(A)$ 
does not depend on the choice of the moduli space. We therefore denote the above function simply with $f$.}
\end{rmk}

\begin{proof}[Proof of Theorem \ref{thm:parallel}]
Let $(X_1,H_1)$ and $(X_2,H_2)$ be two polarized irreducible holomorphic symplectic varieties of $K_n(A)$-type. Suppose, for $i=1,2$, that $(X_i,H_i)$ is polarized deformation equivalent to $(M_{v_i}(A_i), L_i)$, for some smooth and polarized moduli space $(M_{v_i}(A_i), L_i)$. 
By Lemmas \ref{lem:mark} and \ref{lem:moninv} the theorem holds for $(M_{v_1}(A_1), L_1)$ and $(M_{v_2}(A_2), L_2)$. By Lemma \ref{lem:mark} and Remark \ref{rmk:1} the theorem holds for $(X_1,H_1)$ and $(X_2,H_2)$. By Remarks \ref{rmk:2} and  \ref{rmk:3}
there is no dependence on the choices made.  
\end{proof}

\begin{rmk}\label{oss:embedding}
{\em The definition of the above function $f$ on a generic Kummer manifold $X$ requires a canonical choice of an embedding $H^2(X)\rightarrow \Lambda\cong U^4$ up to an isomorphism of $\Lambda$ of determinant one. Such an embedding is given by Mukai's isomorphism for moduli spaces and it is widely expected to be topologically defined, as in the case of manifolds of $K3^{[n]}$ type (cf. \cite[Thm. 9.3]{Markman10}). However, when $n+1$ is a prime power, such an embedding is uniquely determined as all isometries of $H^2(X,\mathbb{Z})$ can be extended to isometries of $\Lambda$ and elements of $\Mon^2(X)$ extend to discriminant one isometries, as in Lemma \ref{lem:moninv}. We have been informed by B. Wieneck that he has obtained such an embedding.}
\end{rmk}

%
\section{Applications}
%
In this section we prove Theorem \ref{thm:existence}. 
\subsection{Explicit polarized deformation equivalence}
Let $h$ be a primitive Hodge class in $H^2(K_n(A), \Z)$, and let $T(h)$ be the saturation in $H^2(K_n(A), \Z)$ of the lattice spanned by $h$ and $v_n=(1,0,-n-1)$. Let $t$ be the \emph{divisibility} of $h$, i.e. the positive integer such that 
$$h \cdot H^2(K_n(A), \Z)=t\Z.$$
Write $h^2=2d$. First we have the following. 

\begin{lem}\label{lem:bounding-coefficients}
There exists an abelian surface $A'$, a primitive polarization $h_{A'}$ on $A'$ and an integer $\beta\in\{0, \ldots, n+1\}$ such that, defining 
$$h'=t h_{A'}-\frac{\beta}{m}\delta\in H^2(K_n(A'), \Z)=H^2(A', \Z)\oplus\Z\delta_n,$$
where $m=gcd(2n+2,\beta)$, the monodromy invariants of $h$ and $h'$ are the same.
\end{lem}

\begin{proof}
To prove the lemma, it suffices to find a polarized parallel transport operator which sends $h$ to $h'$, as the monodromy invariant will stay constant. First of all, let us deform $(A,h_A)$ to a primitively polarized abelian surface $(A',h_{A'})$ such that $U\subset NS(A')$. Notice that  $NS(A')$ contains elements having any (even) square. In this way all the choices that will be made below are possible. 


By abuse of notation, we keep denoting by $h$ the parallel transport of $h$ from $K_n(A)$ to $K_n(A')$.\\   
As a sublattice of $\widetilde \Lambda$, the cohomology group $H^2(K_n(A'), \Z)$ is spanned by $(1, 0, n+1)$ and by $H^2(A', \Z)$. Write 
$$h=(\mu, \lambda h_{A'}, \mu (n+1))$$
where $\lambda$ and $\mu$ are two integers and $h_{A'}$ is primitive. Since $h$ is primitive, $\lambda$ and $\mu$ are relatively prime. Notice that the divisibility of $h$ is 
$$t=\gcd(\lambda, 2n+2).$$

By Remark \ref{rem:discri}, the discriminant group $H^2(K_n(A'),\Z)^\vee/H^2(K_n(A'),\Z)$ is generated by the element $[(1/(2n+2),0,1)]$, which we will denote by $[1]$. 
 If we set $m=(2n+2)/t$, the image of $h/t$ in the discriminant group is therefore $[\mu\cdot m]$, which is of order $t$. Let $\beta$ be the smallest positive integer satisfying $\beta\equiv\mu\cdot m$ modulo $2n+2$. As $[\beta]$ has order $t$, we have that $gcd(\beta,2n+2)=m$.   
 
It is easy to check that $h^2+ (\beta/m)^2 \cdot 2(n+1)$ is divisible by $2t^2$, so if we write
$$
 h^2 = t^2 (2d') - (\beta/m)^2 \cdot 2(n+1)
$$
we choose a primitive polarization $h_{A'}$ on $A'$ of degree $2d'$ (which we can do as 
 $U\subset NS(A')$).
 
By Eichler's criterion \ref{lem:eichler}, we have an isometry $g\in \tilde{O}^+(H^2(K_n(A'),\mathbb{Z}))$ (so in particular $g$ acts trivially on the discriminant group) such that 
$$
g(h)=(\beta/m, th'_{A'},\beta(n+1)/m)=h',
$$ 
for some primitive and ample $h'_{A'}\in NS(A)$. Moreover, this isometry has determinant one. Hence it is a monodromy operator by Theorem \ref{thm:mon}. By definition, it then preserves the monodromy invariant $f$, i.e. 
$$
 f(h)=f(g(h))=f(h').
$$

Now we only need to prove that we can choose $0\leq \beta \leq n+1$, but this is easily done by replacing $h$ with $r_{\delta_n}(h)$, where $r_{\delta_n}$ is the reflection with respect to ${\delta_n}$. Indeed, $r_{\delta_n}$ is a monodromy operator which acts as $-1$ on the discriminant group, therefore if $\beta\geq n+1$, the corresponding coefficient for $r_{\delta_n}(h)$ is smaller than $n+1$. 
\end{proof}

Then, using Theorem \ref{thm:parallel}, proved in Section 3, we obtain the following.

\begin{thm}\label{thm:control-of-polarization}
Let $n>1$ be an integer, and let $(X, h)$ be a primitively polarized irreducible holomorphic symplectic variety of $K_n$-type. Then there exists an abelian surface $A$, a primitive polarization $h_A$ on $A$, an integer $t$, an integer $\beta\in \{0, \ldots, n+1\}$ and $m=gcd(\beta,2n+2)$ such that the pair $(X, h)$ is deformation-equivalent to $(K_n(A), th_A-(\beta/m)\delta)$.
\end{thm}

\begin{proof}
Remark \ref{remark:projective-def-classes} and the local Torelli theorem allow us to assume that $X$ is of the form $K_n(A_0)$, where $A_0$ is an abelian surface (one can alternatively use \cite[Proposition 7.1]{Markman10}). Lemma \ref{lem:bounding-coefficients} shows that there exists a surface $A$, a primitive polarization $h_A$ on $A$ and an integer $\beta\in \{0, \ldots, n+1\}$ and an isomorphism $T(h)\ra T(th_A-\mu/m\delta)$ sending $h$ to $th_A-(\beta/m)\delta$, where $t$ is the divisibility of $h$ and $m=gcd(\beta,2n+2)$. 

By the characterization of the polarized parallel transport operators proved in Section 3,  the condition above ensures that $(X, h)$ and $(K_n(A), th_A-(\beta/m)\delta)$ are deformation equivalent.
\end{proof}

%
\subsection{Examples}\label{sec:ex}
%
Fix integers $n>0$, $2\leq k\leq n+1$ and $g=k$. For any polarised abelian surface $(A,H)$
with $\Pic(A)=\Z H_A$ such that $p_a(H)\geq g+2$, \cite[Theorem 1.1]{KLM} yields
a $g$-dimensional family $\mathcal C\to T,\ T\subset \{H_A\}$, of (singular) curves of geometric genus $g$. We use here the classical notation of $\{H_A\}$ to denote the continuous system  parametrizing curves in the linear system $|H_A|$
and in all of its
translates by points of
$A$. As for our choice of $d$ and $g$ the Brill-Noether number $\rho(g,1,k+1)$ is positive, each curve of the family has a normalization possessing a $\mathfrak g^1_{k+1}$. Inside the relative $(k+1)$-th 
symmetric product $Sym_{k+1} \tilde {\mathcal C}/ T$ of the family $ \tilde {\mathcal C}/ T$ of 
the normalizations consider the family $ \tilde {\mathcal C}^1_{k+1}/ T$ of the $\mathbb P^1$'s associated to the $\mathfrak g^1_{k+1}$'s. We denote by $\mathbb P^1_{\mathfrak g^1_{k+1}}$ the image, inside $A^{[k+1]}$, of the associated rational curve.
Its class can be easily computed, thanks to Riemann-Hurwitz and the intersection pairing in $Pic(K_k(A))=NS(A)\oplus \mathbb Z \delta_k$
(see  e.g. \cite[Section 3]{KLM2}). We have
$$
 [\mathbb P^1_{\mathfrak g^1_{k+1}}]= H_A - 2k r_k.
$$
Take a generic subscheme $\eta\in A^{[n-k+1]}$ such that $\supp(\eta)\cap \supp(\mathbb P^1_{\mathfrak g^1_{k+1}})=\emptyset$. By abuse of notation, we still denote by 
$\mathbb P^1_{\mathfrak g^1_{k+1}}$ the rational curve in $K_n(A)$ obtained by adding to $\mathbb P^1_{\mathfrak g^1_{k+1}}\subset K_k(A)$ the subscheme $\eta$. It is easy to check that its class in $N_1(K_n(A))$ is
 $$
 [\mathbb P^1_{\mathfrak g^1_{k+1}}]= H_A - 2k r_n.
$$
Counting dimensions we also record the following equality
$$
\dim(\tilde {\mathcal C}^1_{k+1}/ T)= \rho(g,1,k+1) + 1 + g= 2k+1.
$$
The fact that for general $t\in T$ we have $\dim(W^1_{k+1}(\tilde C_t))=\rho(g,1,k+1)$ follows e.g. from \cite[Theorem 1.6, (ii)]{KLM}. Notice however that here we need only the inequality  
$\dim(W^1_{k+1}(\tilde C_t))\geq \rho(g,1,k+1)$
which follows from the positivity of $\rho(g,1,k+1)$ and from the classical Brill-Noether theory (see e.g. \cite{ACGH}).
As for the choice $g=k$ the natural map from $ \tilde {\mathcal C}^1_k/ T$ to $A^{[k+1]}$ is finite onto its image (cf. \cite[Example 4.1, 3)]{V15}), we get in this way a uniruled divisor inside $K_k(A)$. 

Finally, by considering for all $k=1,\ldots, n$ the (closure of the) image of the rational map
from $\tilde {\mathcal C}^1_k/ T + A^{(n-k+1)}$ into $A^{[n+1]}$ and restricting it to the fiber of the Albanese map, we obtain $n+1$ different uniruled divisors $D_1,\ldots, D_{n+1}$ inside $K_n(A)$, whose classes are dual, with respect to the Beauville-Bogomolov pairing, to the classes of the rational curves $\mathbb P^1_{\mathfrak g^1_{k+1}}$. 
Precisely, for $k=1,\ldots,n+1$, we have
$$
[D_k]= H_A- \frac{k}{n+1}\delta_n.
$$ 
In particular
$$
q(D_k)=2p_a(H_A)-2 - 2\cdot \frac{k^2}{n+1}= 2(p_a(H_A)-1 - g\cdot \frac{g}{n+1})>0,
$$
for $p_a(H)\geq g+2.$

Notice that this covers all possible monodromy orbits for divisors $h$ such that the class $[h/div(h)]$  in the discriminant group is even. 

To obtain the other cases, we proceed as follows. 
Let $\xi\in K_n(A)$ be a subscheme corresponding to a (simple) ramification point of a linear series
${\mathfrak g}^1_{k+1}$ on the normalization of a curve in $A$. Consider the only exceptional rational curve $\mathbb P^1_\xi$ passing through $\xi$. Let 
\begin{equation}\label{eq:red}
C_k:= \mathbb P^1_{\mathfrak g^1_{k+1}}\cup \mathbb P^1_\xi. 
\end{equation}
Therefore we have
$$
 [C_k]= H_A - (2k-1) r_n.
$$
The dual classes, with respect to the Beauville-Bogomolov pairing, are 
$$
 [D_k]= H_A- \frac{2k-1}{2(n+1)}\delta_n.
$$

Finally, if $\lambda :=gcd (k, n+1)$ (respectively $\lambda :=gcd (2k-1, 2(n+1))$) write 
$$
 k= \lambda \cdot s,\  n+1 = \lambda \cdot t
$$
(respectively $2k-1= \lambda \cdot s\ , 2(n+1) = \lambda \cdot t$).
Then $t$ is the smallest positive integer such that $t [D_k]$ is integral. 
To obtain the case of divisibility one, we can always find an abelian surface $A'$ and a primitive divisor $H'\in NS(A')$ such that  $(K_n(A), H_A-\delta_n)$ 
is deformation equivalent to $(K_n(A'), H')$ .

Notice that, by construction, the classes
$$
 t [D_k] = t H_A - s \delta_n
$$
have divisibility equal to $t$ and cover all the possible polarizations  appearing in Theorem \ref{thm:control-of-polarization}.
\begin{rmk} 
{\em Notice that the classes $D_k$ and $D'_k$ constructed above may not be ample, but as they have positive square, by Huybrechts's projectivity criterion \cite{Huy}, we can always deform $S^{[n]}$ together with $D_k$ (or $D'_k$) to a polarize irreducible holomorphic symplectic variety $(X',h')$. }
\end{rmk}
%
\subsection{Proof of Theorem 1.2}
%

Let $X$ be a smooth projective irreducible holomorphic symplectic variety of dimension $2n$ and let $f : C\ra X$ be a map from a stable curve $C$ of genus zero to $X$. We assume furthermore $f$ is unramified at the generic point of each irreducible component of $C$. Let $\mathcal X\ra B$ be a smooth projective morphism of smooth connected quasi-projective varieties and let $0$ be a point of $B$ such that $\mathcal X_0=X$. Let $\mathcal \alpha$ be a global section of Hodge type $(2n-1,2n-1)$ of $R^{4n-2}\pi_*\Z$ such that $\alpha_0=f_*[C]$ in $H^{4n-2}(X, \Z)$. Consider the relative Kontsevich moduli stack of genus zero stable curves $\overline{{\mathcal M_0}}(\mathcal X/B, \alpha)$, parametrizing maps $f : C\ra X$ from genus zero stable curves to fibers $X=\mathcal X_b$ of $\pi$ such that $f_*[C]=\alpha_b$ (cf. e.g. \cite{AV}). If $f$ is a stable map, we denote by $[f]$ the corresponding point of the Kontsevich moduli stack.

 Let $f:\mathbb P^1\to X$ be a non-constant map. 
Let $ M$ be an irreducible component of $\overline{ M_0}(X, f_*[\mathbb P^1])$ containing $[f]$, and let $Y$ be the subscheme of $X$ covered by the deformations of $f$ parametrized by $ M$. Let $C \to  M$ be the universal curve  
and 
$ev:C\to Y\subset X$ be the induced evaluation map. 
%

\begin{prop}[\cite{CP}, Corollary 3.5]\label{prop:divisor}
Let everything be as above.
If $Y$ is a divisor in $X$, then: 
\begin{enumerate}
\item Any irreducible component of $\overline{\mathcal M_0}(\mathcal X/B, \alpha)$ containing $[f]$ dominates $B$. In particular, the stable map $f : C\ra X$ deforms over a finite cover of $B$;
\item for any point $b$ of $B$, the fiber $X_b$ of $\mathcal X$ over $b$ contains a uniruled divisor.
\end{enumerate}
\end{prop}
\begin{proof}[Proof of Theorem \ref{thm:existence}]
The result follows from Theorem \ref{thm:control-of-polarization}, Proposition \ref{prop:divisor} (and Remark \ref{rmk:red}) and the examples presented in Section \ref{sec:ex}.
\end{proof}
%

\begin{rmk}\label{rmk:red}
{\em To deal with the reduced case it is sufficient to observe that the reduced curves $C_k= \mathbb P^1_{\mathfrak g^1_{k+1}}\cup \mathbb P^1_\xi$
form a complete family of dimension $2n-2$. Then one can invoke 
\cite[Proposition 3.1]{CP}, which is valid also in the reducible case, to conclude that the curves $C_k$ can be deformed along their Hodge locus.
}
\end{rmk}

\end{document}